\documentclass[tikz,border=10pt]{article}
\UseRawInputEncoding
\usepackage[english]{babel}
\usepackage[utf8x]{inputenc}
\usepackage[T1]{fontenc}

\usepackage{tikz}
\usepackage{pgflibraryarrows}
\usepackage{pgflibrarysnakes}
\usetikzlibrary{decorations.markings}

\usepackage[a4paper,top=3cm,bottom=2cm,left=3cm,right=3cm,marginparwidth=1.75cm]{geometry}
\usepackage{amsmath}
\usepackage{graphicx}
\usepackage[colorinlistoftodos]{todonotes}
\usepackage[colorlinks=true, allcolors=blue]{hyperref} \usepackage{mathrsfs}
\usepackage{amssymb}
\usepackage{lineno}
\usepackage{bbding}
\usepackage{fancyhdr}\usepackage{graphics}
\usepackage{titlesec}\usepackage{titletoc}
\usepackage{indentfirst}
\usepackage{amsmath}\usepackage{amsfonts}\usepackage{latexsym}
\usepackage{eucal}\usepackage{epsfig}\usepackage{mathrsfs,amsmath,amsthm}
\usepackage{geometry}

\newtheorem{theorem}{Theorem}[section]
 
 \newtheorem{proposition}[theorem]{Proposition}
 \newtheorem{corollary}[theorem]{Corollary}

 \newcount\refno
\def\CC{{\mathbb C}}

 \def\RR{{\mathbb R}}
 
 \def\SS{{\mathbb S}}

\def\SD{{\mathscr D}}
 \def\SF{{\mathscr F}}
 
 \title{\bf The Real Characterization of $H_{\lambda}^p(\RR_+^2)$  for  $\frac{2\lambda}{2\lambda+1}<p\leq1$
 \thanks{1} \footnote{E-mail:
huzhuoran010@163.com[ZhuoRan Hu].}}

\author{{ZhuoRan Hu}\\
{\small }\\
{\small Beijing 100048, China}}

\begin{document}
\maketitle \setcounter{page}{1} \pagestyle{myheadings}
 \markboth{Hu}{The Real Characterization of $H_{\lambda}^p(\RR_+^2)$  for  $\frac{2\lambda}{2\lambda+1}<p\leq1$}
\begin{abstract}
For $p>p_0=\frac{2\lambda}{2\lambda+1}$ with $\lambda>0$, the Hardy space $H_{\lambda}^p(\RR_+^2)$  associated with the Dunkl transform $\mathcal{F}_\lambda$ and the Dunkl operator $D$ on the real line $\RR$, where $D_xf(x)=f'(x)+\frac{\lambda}{x}[f(x)-f(-x)]$, is the set of functions $F=u+iv$ on the upper half plane $\RR^2_+=\left\{(x, y): x\in\RR, y>0\right\}$, satisfying $\lambda$-Cauchy-Riemann equations: $ D_xu-\partial_y v=0$, $\partial_y u +D_xv=0$, and $\sup\limits_{y>0}\int_{\RR}|F(x+iy)|^p|x|^{2\lambda}dx<\infty$. In this paper, we  will give a further characterization of $H_{\lambda}^p(\RR_+^2)$ in\,\cite{ZhongKai Li 3}. We prove the inequality $\|F\|_{H_{\lambda}^p(\RR_+^2)}\leq c\|u_{\nabla}^*\|_{L^p_{\lambda}}$, which gives a Real Characterization of the class $H_{\lambda}^p(\RR_+^2)$ for $\frac{2\lambda}{2\lambda+1}<p\leq1$ as a main result.

\vskip .2in
 \noindent
 {\bf 2000 MS Classification:} 42B20, 42B25, 42A38.
 \end{abstract}

\setcounter{page}{1}

\section{Introduction}
In\,\cite{ZhongKai Li 3}, a theory about the Hardy spaces on the half-plane $\RR^2_+=\left\{(x, y): x\in\RR, y>0\right\}$ associated with the Dunkl transform on the line $\RR$ is developed. For $0<p<\infty$, $L_{\lambda}^p(\RR)$ is the set of measurable functions satisfying
$ \|f\|_{L_{\lambda}^p}=\Big(c_{\lambda}\int_{\RR}|f(x)|^p|x|^{2\lambda}dx\Big)^{1/p}$ $<\infty$ with
$c_{\lambda}^{-1}=2^{\lambda+1/2}\Gamma(\lambda+1/2)$,
and $p=\infty$ is the usual $L^\infty(\RR)$ space.
For $\lambda\geq0$, The Dunkl operator on the line is\,(cf. \cite{Du3},\cite{Du6}):
$$D_xf(x)=f'(x)+\frac{\lambda}{x}[f(x)-f(-x)]$$
involving a reflection part.
The Dunkl transfrom  for $f\in L_{\lambda}^1(\RR)$ is given by:
\begin{eqnarray}\label{fourier}
(\SF_{\lambda}f)(\xi)=c_{\lambda}\int_{\RR}f(x)E_\lambda(-ix\xi)|x|^{2\lambda}dx,\quad
\xi\in\RR ,
\end{eqnarray}
where $E_{\lambda}(-ix\xi)$ is the Dunkl kernel
$$E_{\lambda}(iz)=j_{\lambda-1/2}(z)+\frac{iz}{2\lambda+1}j_{\lambda+1/2}(z),\ \  z\in\CC$$
and $j_{\alpha}(z)$ is the normalized Bessel function
$$j_{\alpha}(z)=2^{\alpha}\Gamma(\alpha+1)\frac{J_{\alpha}(z)}{z^{\alpha}}=\Gamma(\alpha+1)\sum_{n=0}^{\infty}\frac{(-1)^n(z/2)^{2n}}{n!\Gamma(n+\alpha+1)} .$$
Since $j_{\lambda-1/2}(z)=\cos z$, $j_{\lambda+1/2}(z)=z^{-1}\sin z$. It follows that $E_0(iz)=e^{iz}$, and $\SF_{0}$ agrees with the usual Fourier transform.
We assume $\lambda>0$ in
what follows.

When u and v  satisfy the $\lambda$-Cauchy-Riemann equations:
\begin{eqnarray}\label{a c r0}
\left\{\begin{array}{ll}
                                    D_xu-\partial_y v=0,&  \\
                                    \partial_y u +D_xv=0&
                                 \end{array}\right.
\end{eqnarray}
the function F(z)=F(x,y)=u(x,y)+iv(x,y)\,(z=x+iy)\, is said to be a $\lambda$-analytic function on the upper half plane $\RR^2_+$. It is easy to see that
$F=u+iv$ is $\lambda$-analytic on $\RR^2_+$ if and only if
$$T_{\bar{z}}F\equiv0,\ \ \ \hbox{with}\,T_{\bar{z}}=\frac{1}{2}(D_x+i\partial_y).$$
If $u$ and $v$ are $C^2$ functions and satisfy\,(\ref{a c r0}), then
\begin{eqnarray}\label{1}
(\triangle_{\lambda}u)(x, y)=0,\ \ \hbox{with}\,\triangle_{\lambda}=D_x^2+ \partial_y^2.
\end{eqnarray}
A $C^2$ function $u(x, y)$ satisfying Formula\,(\ref{1}) is said to be $\lambda$-harmonic.\,In\,\cite{ZhongKai Li 3}, the
Hardy space $H^p_\lambda(\RR^2_+)$ for $p>0$ is defined to be the set of
$\lambda$-analytic functions on $\RR^2_+$ satisfying
$$\|F\|_{H^p_\lambda(\RR^2_+)}=\sup\limits_{y>0}\left\{c_{\lambda}\int_{\RR}|F(x+iy)|^p|x|^{2\lambda}dx \right\}^{1/p}<+\infty.$$
When $p>\frac{2\lambda}{2\lambda+1}$, some basic conclusions on $H_{\lambda}^p(\RR_+^2)$  are obtain  \cite{ZhongKai Li 3}, together with
the associated real Hardy space $H_{\lambda}^p(\RR)$ on the line $\RR$, the collection  of the real parts of boundary functions   of  $F\in H_{\lambda}^p(\RR_+^2)$.

 Muckenhoupt and Stein studied the  Hardy spaces associated with the Hankel transform in \cite{MS}.
 Their starting point is the generalized Cauchy-Riemann equations:
\begin{eqnarray}\label{lam-c-r}
u_x-v_y=0,\ \ \ u_y+v_x+\frac{2\lambda}{x}v=0
\end{eqnarray}
for functions $u(x,y)$, $v(x,y)$ on the quarter-plane $\{(x, y): x>0, y>0\}$. The connection with the Hankel transform is as follows. For a function $f\in L^1([0, \infty); x^{2\lambda}dx),$  let $\widehat{f}_{\lambda}(\xi)=2c_{\lambda}\int_0^{\infty}f(x)j_{\lambda-1/2}(x\xi)x^{2\lambda}dx$ be its Hankel transform, then the associated Poisson integral $u(x, y)$ and its conjugate $v(x, y)$  given by
\begin{eqnarray}\label{2}
u(x, y)=2c_\lambda\int_0^{\infty} e^{-y\xi}j_{\lambda-1/2}(x\xi)\widehat{f}_{\lambda}(\xi)\xi^{2\lambda}d\xi,
\end{eqnarray}
\begin{eqnarray}\label{3}
v(x, y)=\frac{2c_\lambda}{2\lambda+1}\int_0^{\infty} e^{-y\xi}j_{\lambda+1/2}(x\xi)x\xi\widehat{f}_{\lambda}(\xi)\xi^{2\lambda}d\xi
\end{eqnarray}
satisfy the system\,(\ref{lam-c-r}). When taking a look to the study of \cite{MS} on the whole line $\RR$, it is restricted to the case that $u$ is even in $x$
and $v$ is odd in $x$, as the functions in (\ref{2})\,and\,(\ref{3}) show, and the non-symmetry of $u$ and $v$ in (\ref{lam-c-r}) leads some ambiguous treatments in further study. It is obvious that for such $u$, $v$, the system\,(\ref{a c r0}) is consistent with\,(\ref{lam-c-r}), and if $f\in L^1_{\lambda}(\RR)$ is even, then $(\SF_{\lambda}f)(\xi)=\widehat{f}_{\lambda}(\xi)$. It is noted that most contents in \cite{MS} are about the theory related to the Gegenbauer  expansions, some of which have been extended to Jacobi expansions in\,\cite{ZhongKai Li 0},\,\cite{ZhongKai Li 1} and to Dunkl-Gegenbauer expansions in\,\cite{ZhongKai Li 2}.

We continue the work in\,\cite{ZhongKai Li 3}. In Section\,\ref{facts}, some facts about the harmonic analysis related to the Dunkl transform on $\RR$ are summarized, including the associated operations such as $\lambda$-translation, $\lambda$-convolution\,$\ast_\lambda$, $\lambda$-Poisson integral, conjugate $\lambda$-Poisson integral and also some fundamental conclusions in \cite{ZhongKai Li 3} on the Hardy space $H_{\lambda}^p(\RR_+^2)$ when $p>\frac{2\lambda}{2\lambda+1}$.

In Section\,\ref{area},   we will give a Real Characterization of  $H^p_{\lambda}(\RR^2_+)$ in Theorem\,\ref{ssst}.\,Our \textbf{main result} in Theorem\,\ref{ssst} shows that the norm of
a $\lambda$-analytic function in $H^p_{\lambda}(\RR^2_+)$  can be characterized by the real-parts which is an analog of the Burkholder-Gundy-Silverstein theorem given in\,\cite{BGS}.

As usual, $\mathcal{B}_\lambda(\RR)$ denotes the set of Borel measures $d\mu$ on $\RR$ for which $\|d\mu\|_{\mathcal{B}_\lambda}=c_\lambda\int_\RR |x|^{2\lambda}|d\mu(x)|$ is finite.  $\SD(\RR)$ designates  the space of $C^{\infty}$ functions on $\RR$ with compact support, $C_c(\RR)$  the spaces of
 continuous function on $\RR$ with compact support, and
$\SS(\RR)$ the space of  $C^{\infty}$ functions on $\RR$ rapidly decreasing together with their derivatives.
$L_{\lambda,{\rm loc}}(\RR)$ is the set of locally integrable functions on $\RR$ associated with the measure $|x|^{2\lambda}dx$. Throughout the paper, the
constants  $c_\lambda$ in\,(\ref{fourier}), $m_\lambda$, $c_\lambda'$,$c_\lambda''$, $\widetilde{c}_\lambda$ and $c_{\lambda,\alpha}$ in the subsequent sections have always the given values respectively, and $c$, $c'$, $c_p$, or $c_m$ denote constants which may be different in different occurrences. In this paper, $|x|^{2\lambda}dx$ on the real line $\RR$, denotes that $|x|^{2\lambda}dx= (x^2)^{\lambda}dx$.
If $\Omega$ is a bounded  domain on the upper half plane, we use $\partial\Omega$ to denote as the boundary of $\Omega$. Let $A$ and $B$ to be two sets on the upper half plane, $A\backslash B$ is the set $A\backslash B=\big\{(x, y)\in \RR_+^2:  (x, y)\in A,\,\hbox{and},\,(x, y)\notin B \big\}$.

\section{Basic Facts  in  Dunkl setting on the real line $\RR$}\label{facts}

\subsection{The Dunkl Operator and the Dunkl Transform}
For sake of simplicity,  we set $|E|_{\lambda}=c_\lambda\int_E |x|^{2\lambda}dx$ for a measurable set $E\subset\RR$, and $\langle f, g\rangle_\lambda=c_\lambda\int_\RR f(x)g(x) |x|^{2\lambda}dx$ whenever the integral exists. The Dunkl operator on the real line is given by
$$Df(x)=f'(x)+\frac{\lambda}{x}[f(x)-f(-x)].$$
A direct computation shows that
$$
(D^2f)(x)=f''(x)+\frac{2\lambda}{x}f'(x)-\frac{\lambda}{x^2}
[f(x)-f(-x)].
$$
For $f\in C^1(\RR)$, an inverse operator of $D$ is given by
\begin{eqnarray}\label{D-operator-inverse-1}
f(x)=f(0)+\frac{x}{2}\left(\int_{-1}^1({\rm sgn}\,s)(Df)(sx)ds+
\int_{-1}^1(Df)(sx)|s|^{2\lambda}ds\right).
\end{eqnarray}
For $f\in \SS(\RR)$, the inverse operator of $D$ can be given by
\begin{eqnarray}\label{D-operator-inverse-2}
f(x)=\frac{1}{2}\int_{-\infty}^{\infty}[\tau_x(Df)](-t)({\rm
sgn}\,t)dt,
\end{eqnarray}
and from the Formula\,$Df(x)=f'(x)+\lambda\int_{|s|\leq1}f'(xs)ds=f'(x)-\lambda\int_{|s|\geq1}f'(xs)ds$, we could know that the operator $D$ is invariant in $\SD(\RR)$ and $\SS(\RR)$.
If $f,\phi\in \SS(\RR)$, or $f\in C^1(\RR)$ with $\phi\in \SD(\RR)$,
or $f\in C_c^1(\RR)$  with $\phi\in C^{\infty}(\RR)$, we could get
\begin{eqnarray}\label{D-operator-dual-1}
\langle Df,\phi\rangle_{\lambda}=-\langle f,D\phi\rangle_{\lambda}.
\end{eqnarray}
A Laplace-type representation of the Dunkl kernel $E_\lambda(iz)$ is given by\,(\cite{ZhongKai Li 3},\cite{Ro1})
\begin{eqnarray}\label{D-kernel-3}
E_\lambda(iz)=c_{\lambda}'\int_{-1}^1e^{izt}(1+t)(1-t^2)^{\lambda-1}dt,
\ \ \ \ \ \
c'_{\lambda}=\frac{\Gamma(\lambda+1/2)}{\Gamma(\lambda)\Gamma(1/2)},
\end{eqnarray}
and $x\mapsto E_{\lambda}(ix\xi)$ is an eigenfunction of $D_x$ with eigenvalue $i\xi$\,(\cite{Du6}), that is
\begin{eqnarray}\label{D-kernel-2}
D_x[E_{\lambda}(ix\xi)]=i\xi E_{\lambda}(ix\xi), \ \ \ \ \ \
\ \ \ E_{\lambda}(ix\xi)|_{x=0}=1.
\end{eqnarray}
From Formulas\,(\ref{D-kernel-3},\ref{D-kernel-2}), an immediately case is that $e^{-y\xi}E_{\lambda}(ix\xi)$ is a $\lambda$-analytic function for a fixed given $\xi$,\  i.e., $T_{\bar{z}}e^{-y\xi}E_{\lambda}(ix\xi)=0$.

The Dunkl transform shares many  important  properties with the usual Fourier transform, part of which are listed as follows. These conclusions extend those on the Hankel transform and are special cases on the general Dunkl transform studied in\,(\cite{dJ},\cite{Du6}).

\begin{proposition}\label{D-transform-a}

{\rm(i)} \ If $f\in L_{\lambda}^1(\RR)$, then $\SF_{\lambda}f\in
C_0(\RR)$ and $\|\SF_{\lambda}f\|_{\infty}\leq\|f\|_{L_{\lambda}^1}$.

{\rm (ii)} \ {\rm (Inversion)} \ If $f\in
L_{\lambda}^1(\RR)$ such that $\SF_{\lambda}f\in L_{\lambda}^1(\RR)$, then
$f(x)=[\SF_{\lambda}(\SF_{\lambda}f)](-x)$.

{\rm(iii)} \ For $f\in\SS(\RR)$ or $f\in C_c^1(\RR)$, we have  $[\SF_{\lambda}(Df)](\xi)=i\xi(\SF_{\lambda}f)(\xi)$, $[\SF_{\lambda}(xf)](\xi)=i[D_{\xi}(\SF_{\lambda}f)](\xi)$ for $\xi\in\RR$, and $\SF_{\lambda}$ is a topological automorphism on $\SS(\RR)$.

{\rm(iv)} \ {\rm (Product Formula)} \ For all $f,g\in L_{\lambda}^1(\RR)$, we have
$\langle \SF_{\lambda}f,g\rangle_{\lambda}=\langle
f,\SF_{\lambda}g\rangle_{\lambda}$.

{\rm (v)} \ {\rm (Plancherel)} There exists a unique extension of
$\SF_{\lambda}$ to $L_{\lambda}^2(\RR)$ with
$\|\SF_{\lambda}f\|_{L_{\lambda}^2}=\|f\|_{L_{\lambda}^2}$, and
$(\SF_{\lambda}^{-1}f)(x)=(\SF_{\lambda}f)(-x)$.
\end{proposition}

\begin{corollary}\label{Hausdorff-Young} {\rm (Hausdorff-Young)}
For $1\le p\le2$, $\SF_{\lambda}f$ exists when $f\in L_{\lambda}^p(\RR)$, and

$\|\SF_{\lambda}f\|_{L_{\lambda}^{p'}}\le\|f\|_{L_{\lambda}^p}$,
where $1/p+1/p'=1$.
\end{corollary}

\subsection{The $\lambda$-Translation and the $\lambda$-Convolution}
For $x,t,z\in\RR$, we set
$W_{\lambda}(x,t,z)=W_{\lambda}^0(x,t,z)(1-\sigma_{x,t,z}+\sigma_{z,x,t}+\sigma_{z,t,x})$,
where
$$
W_{\lambda}^0(x,t,z)=\frac{c''_{\lambda}|xtz|^{1-2\lambda}\chi_{(||x|-|t||,
|x|+|t|)}(|z|)} {[((|x|+|t|)^2-z^2)(z^2-(|x|-|t|)^2)]^{1-\lambda}},
$$
$c''_{\lambda}=2^{3/2-\lambda}\big(\Gamma(\lambda+1/2)\big)^2/[\sqrt{\pi}\,\Gamma(\lambda)]$,
and  $\sigma_{x,t,z}=\frac{x^2+t^2-z^2}{2xt}$ for $x,t\in
\RR\setminus\{0\}$,
and 0 otherwise. From\,\cite{Ro1}, we have

\begin{proposition}\label{Dunkl-kernel-P}
The Dunkl kernel $E_\lambda$ satisfies the following product formula:
\begin{eqnarray}\label{product-D-1}
E_\lambda(x\xi)E_\lambda(t\xi)=\int_{\RR}E_\lambda(z\xi)d\nu_{x,t}(z),\quad
x,t\in\RR \ ~\hbox{and} ~ \ \xi\in\CC,
\end{eqnarray}
where $\nu_{x,t}$ is a signed measure given by
$d\nu_{x,t}(z)=c_{\lambda}W_\lambda(x,t,z)|z|^{2\lambda}dz$
for $x,t\in \RR\setminus\{0\}$, $d\nu_{x,t}(z)=d\delta_x(z)$ for $t=0$,
 and $d\nu_{x,t}(z)=d\delta_t(z)$ for $x=0$.
\end{proposition}
If $t\neq0$, for an appropriate function $f$ on $\RR$, the  $\lambda$-translation is given by
\begin{eqnarray}\label{tau-1}
(\tau_tf)(x)=c_{\lambda}\int_{\RR}f(z)W_\lambda(x,t,z)|z|^{2\lambda}dz;
\end{eqnarray}
and if $t=0$, $(\tau_0f)(x)=f(x)$.
If $(\tau_tf)(x)$ is taken as a function of $t$ for a given $x$,
we may set $(\tau_tf)(0)=f(t)$ as  a complement.
An unusual fact is that
$\tau_t$ is not a positive operator in general \cite{Ro1}.
If $(x, t)\neq(0,0)$, an equivalent form of $(\tau_tf)(x)$ is given by\,\cite{Ro1},
\begin{eqnarray}\label{tau-2}
(\tau_tf)(x)=c'_{\lambda}\int_0^\pi \bigg(f_e(\langle
x,t\rangle_\theta)+ f_o(\langle
x,t\rangle_\theta)\frac{x+t}{\langle
x,t\rangle_\theta}\bigg)(1+\cos\theta)\sin^{2\lambda-1}\theta
d\theta
\end{eqnarray}
where $x,t\in\RR$, $f_e(x)=(f(x)+f(-x))/2$,
$f_o(x)=(f(x)-f(-x))/2$, $\langle
x,t\rangle_\theta=\sqrt{x^2+t^2+2xt\cos\theta}$.
For two appropriated function $f$ and $g$, their $\lambda$-convolution
 $f\ast_{\lambda}g$ is defined by
\begin{eqnarray}\label{convolution-10}
(f\ast_{\lambda}
g)(x)=c_{\lambda}\int_{\RR}(\tau_xf)(-t)g(t)|t|^{2\lambda}dt.
\end{eqnarray}
The properties of $\tau$ and $\ast_{\lambda}$ are listed as follows\,\cite{ZhongKai Li 3}:

\begin{proposition}\label{tau-convolution-a} {\rm(i)} \ If $f\in L_{\lambda,{\rm loc}}(\RR)$,
then for all $x,t\in\RR$, $(\tau_tf)(x)=(\tau_xf)(t)$, and
$(\tau_t\tilde{f})(x)=(\widetilde{\tau_{-t}f})(x)$, where $\widetilde{f}(x)=f(-x)$.

{\rm(ii)}  \ For all $1\le p\le\infty$ and $f\in L_{\lambda}^p(\RR)$,
$\|\tau_tf\|_{L^p_{\lambda}}\le 4\|f\|_{L^p_{\lambda}}$ with $t\in\RR$, and for $1\leq p<\infty$, $\lim_{t\rightarrow0}\|\tau_tf-f\|_{L^p_{\lambda}}=0$.

{\rm(iii)} \ If $f\in L^p_\lambda(\RR)$, $1\leq p\leq2$ and $t\in\RR$, then
$[\SF_{\lambda}(\tau_t
f)](\xi)=E_{\lambda}(it\xi)(\SF_{\lambda}f)(\xi)$ for almost every $\xi\in\RR$.

{\rm(iv)} \ For measurable $f,g$ on $\RR$, if $\iint
|f(z)||g(x)||W_{\lambda}(x,t,z)||z|^{2\lambda}|x|^{2\lambda}dzdx$ is convergent, we have
  $\langle\tau_tf,g\rangle_{\lambda}=\langle
f,\tau_{-t}g\rangle_{\lambda}$. In particular,
$\ast_{\lambda}$ is commutative.

{\rm(v)} \ {\rm(Young inequality)} \ If $p,q,r\in[1,\infty]$ and
$1/p+1/q=1+1/r$, then for  $f\in L^p_{\lambda}(\RR)$, $g\in
L^q_{\lambda}(\RR)$, we have $\|f\ast_{\lambda}g\|_{L^r_{\lambda}}\leq
4\|f\|_{L^p_{\lambda}} \|g\|_{L^q_{\lambda}}$.

{\rm(vi)} \ Assume that $p,q,r\in[1,2]$, and $1/p+1/q=1+1/r$. Then for $f\in L^p_{\lambda}(\RR)$ and $g\in
L^q_{\lambda}(\RR)$,
$[\SF_{\lambda}(f\ast_{\lambda}g)](\xi)=(\SF_{\lambda}f)(\xi)(\SF_{\lambda}g)(\xi)$. In particular,
$\ast_{\lambda}$ is associative in $L^1_{\lambda}(\RR)$

{\rm(vii)} \ If $f\in L^1_{\lambda}(\RR)$ and $g\in\SS(\RR)$, then
$f\ast_{\lambda}g\in C^{\infty}(\RR)$.

{\rm(viii)} \ If $f,g\in L_{\lambda,{\rm loc}}(\RR)$, ${\rm
supp}\,f\subseteq\{x:\, r_1\le|x|\le r_2\}$ and ${\rm
supp}\,g\subseteq\{x:\,|x|\le r_3\}$,$r_2>r_1>0$, $r_3>0$, then
${\rm supp}\, (f\ast_{\lambda}g)\subseteq\{x:\,r_1-r_3\le|x|\le
r_2+r_3\}$.
\end{proposition}

\begin{proposition}(\hbox{\cite{ZhongKai Li 3}, Corollary\,2.5(i),\,Lemma\,2.7,\,Proposition\,2.8(ii)})\label{D-translation-a}

{\rm (i)} \ If $f\in\SS(\RR)$ or $\SD(\RR)$, then for fixed
$t$, the function
$x\mapsto(\tau_tf)(x)$ is also in $\SS(\RR)$ or $\SD(\RR)$, and
\begin{eqnarray}\label{D-translation-10}
 D_t(\tau_tf(x))=D_x(\tau_tf(x))=[\tau_t(Df)](x).
\end{eqnarray}
If $f\in\SS(\RR)$ and $m>0$, a pointwise estimate is given by
\begin{eqnarray}\label{D-translation-2}
\big|(\tau_t|f|)(x)\big|\le
\frac{c_m(1+x^2+t^2)^{-\lambda}}{(1+||x|-|t||^2)^m}.\end{eqnarray}

{\rm (ii)} \ If $\phi\in L_\lambda^1(\RR)$ satisfies $(\SF_{\lambda}\phi)(0)=1$ and $\phi_{\epsilon}(x)=\epsilon^{-2\lambda-1}\phi(\epsilon^{-1}x)$ for
$\epsilon>0$, then for all $f\in X= L_{\lambda}^p(\RR)$, $ 1\leq p<\infty$, or $C_0(\RR)$, $\lim_{\epsilon\rightarrow0+}\|f\ast_{\lambda}\phi_{\epsilon}-f\|_X=0$.

\end{proposition}

\subsection{The $\lambda$-Poisson Integral and its Conjugate Integral}
The following  results about  $\lambda$-Poisson Integral and Conjugate $\lambda$-Poisson  Integral are obtained in \cite{ZhongKai Li 3}.
For $1\leq p\leq\infty$, the $\lambda$-Poisson integral of $f\in L_\lambda^p(\RR)$ in the Dunkl setting is given by:
$$
(Pf)(x,y)=c_{\lambda}\int_{\RR}f(t)(\tau_xP_y)(-t)|t|^{2\lambda}dt,
\ \ \ \ \ \ \hbox{for} \ x\in\RR, \ y\in(0,\infty),
$$
where $(\tau_xP_y)(-t)$ is the $\lambda$-Poisson kernel with $P_y(x)=m_{\lambda}y(y^2+x^2)^{-\lambda-1}$,
$m_\lambda=2^{\lambda+1/2}\Gamma(\lambda+1)/\sqrt{\pi}$. Similarly,  the $\lambda$-Poisson integral for  $d\mu\in {\frak
B}_{\lambda}(\RR)$ can be given by $
(P(d\mu))(x,y)=c_{\lambda}\int_{\RR}(\tau_xP_y)(-t)|t|^{2\lambda}d\mu(t). $

\begin{proposition}\cite{ZhongKai Li 3}\label{Poisson-a} {\rm(i)} \
$\lambda$-Poisson kernel $(\tau_xP_y)(-t)$ can be represented by
\begin{eqnarray}\label{D-Poisson-ker-11}
(\tau_xP_y)(-t)=
\frac{\lambda\Gamma(\lambda+1/2)}{2^{-\lambda-1/2}\pi}\int_0^\pi\frac{y(1+{\rm
sgn}(xt)\cos\theta)
}{\big(y^2+x^2+t^2-2|xt|\cos\theta\big)^{\lambda+1}}\sin^{2\lambda-1}\theta
d\theta.
\end{eqnarray}

{\rm(ii)} \
The Dunkl transform of the function $P_y(x)$ is $(\SF_{\lambda}P_y)(\xi)=e^{-y|\xi|}$, and
$
(\tau_xP_y)(-t)=c_{\lambda}\int_{\RR}e^{-y|\xi|}E_{\lambda}(ix\xi)E_{\lambda}(-it\xi)|\xi|^{2\lambda}d\xi.
$
\end{proposition}
Relating to $P_y(x)$, the conjugate $\lambda$-Poisson integral for $f\in L^p_{\lambda}(\RR)$ is given by:
$$
(Qf)(x,y)=c_{\lambda}\int_{\RR}f(t)(\tau_xQ_y)(-t)|t|^{2\lambda}dt,
\ \ \ \ \ \ \hbox{for} \ x\in\RR, \ y\in(0,\infty),
$$
where $(\tau_xQ_y)(-t)$ is the conjugate $\lambda$-Poisson kernel with $Q_y(x)=m_{\lambda}x(y^2+x^2)^{-\lambda-1}$.

\begin{proposition}\cite{ZhongKai Li 3}\label{conjugate-Poisson-a} {\rm(i)} \ \
The conjugate  $\lambda$-Poisson kernel $(\tau_xQ_y)(-t)$ can be represented by
\begin{eqnarray*}\label{D-conjugate-Poisson-ker-1}
(\tau_xQ_y)(-t)=
\frac{\lambda\Gamma(\lambda+1/2)}{2^{-\lambda-1/2}\pi}\int_0^\pi\frac{(x-t)(1+{\rm
sgn}(xt)\cos\theta)
}{\big(y^2+x^2+t^2-2|xt|\cos\theta\big)^{\lambda+1}}\sin^{2\lambda-1}\theta
d\theta.
\end{eqnarray*}

{\rm(ii)} \  The Dunkl transform of $Q_y(x)$ is given by  $(\SF_{\lambda}Q_y)(\xi)=-i({\rm
sgn}\,\xi)e^{-y|\xi|}$, for $\xi\neq0$, and
$
(\tau_xQ_y)(-t)=-ic_{\lambda}\int_{\RR}({\rm
sgn}\,\xi)e^{-y|\xi|}E_{\lambda}(ix\xi)E_{\lambda}(-it\xi)|\xi|^{2\lambda}d\xi.
$
\end{proposition}

Then we can define the associated maximal functions as
$$(P^*_{\nabla}f)(x)=\sup_{|s-x|<y}|(Pf)(s,y)|,\ \ \
(P^*f)(x)=\sup_{y>0}|(Pf)(x,y)|,$$
$$(Q^*_{\nabla}f)(x)=\sup_{|s-x|<y}|(Qf)(s,y)|,\ \ \ (Q^*f)(x)=\sup_{y>0}|(Qf)(x,y)|.$$

\begin{proposition}\cite{ZhongKai Li 3}\label{Poisson-conjugate-CR}
{\rm (i)}    For $f\in L^p_{\lambda}(\RR)$, $1\le p<\infty$,  $u(x,y)=(Pf)(x,y)$ and  $v(x,y)=(Qf)(x,y)$
on $\RR^2_+$ satisfy the $\lambda$-Cauchy-Riemann equations\,(\ref{a c r0}), and are both $\lambda$-harmonic on $\RR^2_+$.\\
{\rm (ii)}\rm(semi-group property)  If $f\in L^p_{\lambda}(\RR)$, $1\leq p\leq \infty$, and $y_0>0$, then
$(Pf)(x,y_0+y)=P[(Pf)(\cdot,y_0)](x,y), \ \hbox{for} \ y>0.$\\
{\rm (iii)}  If $1 < p < \infty$, then there exists some constant $A_p'$ for any $f \in
L^p_{\lambda}(\RR)$, $\|(Q^*_{\nabla})f\|_{L^p_{\lambda}} \le A_p'
\|f\|_{L^p_{\lambda}}$.\\
{\rm (iv)}   $P^*_{\nabla}$ and $P^*$ are both $(p,p)$ type for $1<p
\leq\infty$ and weak-$(1,1)$ type.\\
{\rm (v)}    If $d\mu\in {\frak B}_{\lambda}(\RR)$, then
$\|(P(d\mu))(\cdot,y)\|_{L_{\lambda}^1}\le\|d\mu\|_{{\frak
B}_{\lambda}}$ as $y\rightarrow0+$, $[P(d\mu)](\cdot,y)$
converges $*$-weakly to $d\mu$: If $f\in X=L_{\lambda}^p(\RR)$, $1\le
p<\infty$, or $C_0(\RR)$, then $\|(Pf)(\cdot,y)\|_X\le\|f\|_X$ and
$\lim_{y\rightarrow0+}\|(Pf)(\cdot,y)-f\|_{X}=0$.\\
{\rm (vi)}  If $f\in L^p_{\lambda}(\RR)$, $1\le p\le2$, and
$[\SF_{\lambda}(Pf(\cdot,y))](\xi)=e^{-y|\xi|}(\SF_{\lambda}f)(\xi)$,
and
\begin{eqnarray}\label{D-Poisson-2}
(Pf)(x,y)=c_{\lambda}\int_{\RR}e^{-y|\xi|}(\SF_{\lambda}f)(\xi)E_{\lambda}(ix\xi)|\xi|^{2\lambda}d\xi,\quad
(x,y)\in\RR^2_+,
\end{eqnarray}
furthermore, Formula\,(\ref{D-Poisson-2}) is true when we replace $f\in L^p_{\lambda}(\RR)$ with $d\mu\in {\frak
B}_{\lambda}(\RR)$.
\end{proposition}

\subsection{The Theory of the Hardy space $H_{\lambda}^p(\RR_+^2)$ }
 The main results about the Hardy spaces $H_{\lambda}^p(\RR_+^2)$ obtained in\,\cite{ZhongKai Li 3} are listed as the following  Theorems\,\ref{majorization-2}--\ref{p5} in this section.

\begin{theorem} \label{majorization-2} {\rm ($\lambda$-harmonic majorization)}
For $F\in H^p_\lambda(\RR^2_+)$ and $p_0\le
p<\infty$ with $p_0=\frac{2\lambda}{2\lambda+1}$, there exists a nonnegative function $g\in
L^{p/p_0}_{\lambda}(\RR)$ such that for $(x, y)\in\RR^2_+$
\begin{align*}
& |F(x,y)|^{p_0}\leq (Pg)(x,y), \\
& \|F\|^{p_0}_{H^{p}_{\lambda}}\le\|g\|_{L^{p/p_0}_{\lambda}} \le
c_p\|F\|^{p_0}_{H^{p}_{\lambda}},
\end{align*}
where $c_p=2^{1-p_0+p_0/p}$ for $p_0\le p<1$ and $c_p=2$ for $p\ge1$.
\end{theorem}

\begin{theorem} \label{ss}
For $p\ge p_0 =\frac{2\lambda}{2\lambda+1}$ and $F\in H_{\lambda}^p(\RR^2_+)$. Then\\
{\rm (i)} \  For a.e. $x\in\RR$, $\lim F(t,y)=F(x)$ exists as $(t,y)$
tends to $(x,0)$ nontangentially.\\
{\rm (ii)} \
If $F=0$ in a subset  $E\subset\RR$ with $|E|_{\lambda}>0$ symmetric about $x=0$, then $F\equiv0$.\\
{\rm (iii)} \  If $p>p_0$, then
$\lim_{y\rightarrow0+}\|F(\cdot,y)-F\|_{L^p_{\lambda}}=0$,
and $\|F\|_{H^{p}_{\lambda}}\approx\left(c_{\lambda}\int_{\RR}|F(x)|^{p}|x|^{2\lambda}dx\right)^{1/p}$.\\
{\rm (iv)} \  For $p >p_0$, $F\in H_\lambda^{p}(\RR^2_+)$ if and only if
$F^*_{\nabla}\in L^{p}_{\lambda}(\RR)$ and
   $ \|F\|_{H^{p}_{\lambda}}\approx \|F^*_{\nabla}\|_{L^p_{\lambda}}$, where   $F_{\nabla}^*(x)=\sup_{|x-u|<y}|F(u, y)|$ is the non-tangential maximal functions of $F$.\\
{\rm (v)} \   If $p >p_0$ and $p_1\geq p_0$, $F(x,y)\in H^{p}_{\lambda}(\RR^2_+)$
and $F(x) \in L^{p_1}_{\lambda}(\RR)$, then $F\in
H^{p_1}_{\lambda}(\RR^2_+)$.
\end{theorem}

\begin{theorem}\label{analytic-character-1}
{\rm (i)} \    If $1\le p<\infty$ and  $F=u+iv\in H_{\lambda}^p(\RR^2_+)$, then $F$
is the $\lambda$-Poisson integral of its boundary values $F(x)$, and $F(x)\in
L^p_{\lambda}(\RR)$.\\
{\rm (ii)} \  If $1\le p\le2$, then the $\lambda$-Poisson integral of a function in $L^{p}_{\lambda}(\RR)$ is in $H^p_{\lambda}(\RR^2_+)$ if
and  only if the Dunkl transform of the function vanishes on $(-\infty, 0)$.\\
{\rm (iii)} \  If $p_0\le p\le1$ and $F(x,y)\in H_{\lambda}^p(\RR^2_+)$,
then $F$ has the following representation
\begin{eqnarray}\label{D-Poisson-5}
F(x,y)=c_{\lambda}\int_0^{\infty}e^{-y\xi}\phi(\xi)E_{\lambda}(ix\xi)|\xi|^{2\lambda}d\xi,\ \ \ \ (x, y)\in\RR_+^2,
\end{eqnarray}
where $\phi$ is a continuous function on $\RR$ satisfying $\phi(\xi)=0$ for $\xi\in(-\infty, 0]$ and that for each $y>0$, the function
$\xi\rightarrow e^{-y|\xi|}\phi(\xi)$ is bounded on $\RR$ and in  $L^{1}_{\lambda}(\RR)$.\\
{\rm (iv)} \  For $F\in H^p_{\lambda}(\RR^2_+)$, $p\ge p_0$,
there is a constant $c>0$, such that  $|F(x,y)|\leq
cy^{-(2\lambda+1)/p}\|F\|_{H^p_\lambda}$ for  $y>0$.\\
{\rm (v)} \  If $p_0<p\le2$,  $F\in H^p_{\lambda}(\RR^2_+)$, then $\int_0^{\infty}|(\SF_{\lambda}F)(\xi)|^p|\xi|^{(2\lambda+1)(p-2)+2\lambda}d\xi\le
c\|F\|_{H_{\lambda}^p}^p,$  where $c$ is a constant independent on $F$.\\
{\rm (vi)} \ For $F\in H^p_{\lambda}(\RR^2_+)$, where $p_0<p<1$ and $k\ge p$,  we could have $$\int_0^{\infty}|(\SF_{\lambda}F)(\xi)|^k|\xi|^{(2\lambda+1)(k-1-k/p)+2\lambda}d\xi\le
c\|F\|_{H_{\lambda}^p}^p,$$  $|(\SF_{\lambda}F)(\xi)|\le
c\xi^{(1+2\lambda)(p^{-1}-1)}\|F\|_{H_{\lambda}^p}$, and $(\SF_{\lambda}F)(\xi)=o\left(\xi^{(1+2\lambda)(p^{-1}-1)}\right)$
as $\xi\rightarrow+\infty$.
\end{theorem}

\begin{theorem} \label{p5}
{\rm (i)} \  For $\frac{2\lambda}{2\lambda+1}< p< l\leq+\infty$,   $\delta=\frac{1}{p}-\frac{1}{l}$,  $F(x, y)\in H^p_{\lambda}(\RR^2_+)$,
\begin{eqnarray}\label{i 2}
\left(\int_{\RR}|F(x,y)|^{l}|x|^{2\lambda}dx \right)^{\frac{1}{l}} \leq cy^{-\delta(1+2\lambda)}\|F\|_{H_{\lambda}^p(\RR_+^2)}.
\end{eqnarray}
{\rm (ii)} \ For $\frac{2\lambda}{2\lambda+1}< p< l\leq+\infty$,   $\delta=\frac{1}{p}-\frac{1}{l}$,
 $F(x, y)\in H^p_{\lambda}(\RR^2_+)$, $$\|F(.,y)\|_{L_{\lambda}^l}=o\left(y^{-\delta(2\lambda+1)}\right),\ \ \hbox{as}\,\ y\rightarrow0+.$$
{\rm (iii)} \ For $\frac{2\lambda}{2\lambda+1}< p< l\leq+\infty$, $p\leq k<\infty$,  $\delta=\frac{1}{p}-\frac{1}{l}$,
and $F(x, y)\in H^p_{\lambda}(\RR^2_+)$,
\begin{eqnarray}\label{i 1}
\left(\int_0^{+\infty} y^{k\delta(1+2\lambda)-1}\left(\int_{\RR}|F(x,y)|^{l}|x|^{2\lambda}dx \right)^{\frac{k}{l}}dy \right)^{\frac{1}{k}}\leq c\|F\|_{H_{\lambda}^p(\RR_+^2)}.
\end{eqnarray}
\end{theorem}

\section{Real Characterization of   functions in $H_{\lambda}^p(\RR_+^2)$  }\label{area}
In this Section,   we will give a Real Characterization of  $H^p_{\lambda}(\RR^2_+)$ in Theorem\,\ref{ssst}, from which we extends the results in\,\cite{ZhongKai Li 3} that $H^p_\lambda(\RR^2_+)$ can be characterized by its real parts.

\begin{proposition}
Let $\Omega$ to be a bounded  domain,  symmetric in x on the upper half plane: $\Omega=\left\{(x, y):\,y>0,\ \  (x, y)\in\Omega\Leftrightarrow (-x, y)\in\Omega\right\}$. $F(z)=F(x,y)=u(x,y)+iv(x,y)$ is a $\lambda$-analytic function, where u and v are real $C^2$ functions satisfying $\lambda$-Cauchy-Riemann equations\,(\ref{a c r0})\,and u(x, y) is  odd or even  in x. Then we could obtain:
\begin{eqnarray}\label{i}
\int_{\partial\Omega}F(z)^2|x|^{2\lambda}dz=0
\end{eqnarray}
\end{proposition}

\begin{proof}
From\,\cite{ZhongKai Li 3}, we could know that the following Formula holds if g(z) is a $\lambda$-analytic function
 \begin{eqnarray}\label{ii1}
\int_{\partial\Omega}g(z)|x|^{2\lambda}dz=0.
\end{eqnarray}
Notice that   $F(z)^2 =u^2(x, y)-v^2(x, y)+2iu(x, y)v(x, y)$ may not  be a $\lambda$-analytic function when F(z) is  $\lambda$-analytic. Thus  Formula\,(\ref{i}) can not be deduced from Formula\,(\ref{ii1}) which is not analog to the properties of  analytic function  in Euclidean spaces. \\If u(x, y) is even in x  and v(x, y)  odd in x, then by Formula\,(\ref{a c r0}) we could obtain
\begin{eqnarray}\left\{\begin{array}{ll}\label{a c r1}
                                    D_x(u^2-v^2)-\partial_y (2uv)=2u(\partial_x u-\partial_y v)-2v(\partial_x v+\partial_y u)=\frac{4\lambda}{x}v^2,&  \\\\\\
                                    \partial_y (u^2-v^2) +D_x(2uv)=2u(\partial_xv+\partial_yu)+2v(-\partial_yv+\partial_xu)+\frac{4\lambda}{x}uv=0.&
\end{array}\right.
\end{eqnarray}
If u(x, y) is odd in x and v(x, y)  even in x, then by Formula\,(\ref{a c r0}) the following equations could be achieved:
\begin{eqnarray}\left\{\begin{array}{ll}\label{a c r2}
                                    D_x(u^2-v^2)-\partial_y (2uv)=2u(\partial_x u-\partial_y v)-2v(\partial_x v+\partial_y u)=-\frac{4\lambda}{x}u^2,&  \\\\\\
                                    \partial_y (u^2-v^2) +D_x(2uv)=2u(\partial_xv+\partial_yu)+2v(-\partial_yv+\partial_xu)+\frac{4\lambda}{x}uv=0.&
                                 \end{array}\right.
\end{eqnarray}
We use $\Omega^+$ and $\Omega^-$ to  denote $\Omega^+ =\{(x, y)|  (x, y)\in\Omega, x\geq0\},\, \Omega^- =\{(x, y)|  (x, y)\in\Omega, x\leq0\}$.
Let $G(z)=G(x, y)=u_1(x, y)+iv_1(x, y)$ to be a $C^1$ function which may not be  $\lambda$-analytic  on $\RR_+^2$, and satisfies that $u_1(x, y)$ and $v_1(x, y)$ are both odd or even in $x$. From Stokes-Theorem we could obtain
\begin{equation}\label{u5}\begin{split}
& \int_{\partial\Omega^+}G(z)|x|^{2\lambda}dz\\&=\int_{\partial\Omega^+}\bigg(u_1(x, y)+iv_1(x, y)\bigg)(x^2)^{\lambda}(dx+idy)\\
&=\int_{\Omega^+} \bigg\{-\bigg(\partial_y u_1+\partial_xv_1+2(\lambda/x)v_1 \bigg)
+i\bigg(\partial_xu_1+2(\lambda/x)u_1-\partial_yv_1 \bigg)\bigg\}|x|^{2\lambda}(dx\wedge dy),\\
\end{split}\end{equation}
and
\begin{equation}\label{u1}\begin{split}
& \int_{\partial\Omega^-}G(z)|x|^{2\lambda}dz\\&=\int_{\partial\Omega^-}\bigg(u_1(x, y)+iv_1(x, y)\bigg)(x^2)^{\lambda}(dx+idy)\\
&=\int_{\Omega^-} \bigg\{-\bigg(\partial_yu_1+\partial_xv_1+2(\lambda/x)v_1 \bigg)
+i\bigg(\partial_xu_1+2(\lambda/x)u_1-\partial_yv_1 \bigg)\bigg\}|x|^{2\lambda}(dx\wedge dy).\\
\end{split}\end{equation}
Notice that $\Omega$ is a bounded  domain symmetric in x, with the fact that $u_1(x, y)$ and $v_1(x, y)$ are both odd or even in $x$, then we could obtain the following Formula\,(\ref{intergal}) from Formulas\,(\ref{u5},\,\ref{u1}):
\begin{equation}\begin{split}\label{intergal}
\int_{\partial\Omega}G(z)|x|^{2\lambda}dz&=\int_{\partial\Omega^+}G(z)|x|^{2\lambda}dz+\int_{\partial\Omega^-}G(z)|x|^{2\lambda}dz\\
&=\int_{\Omega} \bigg\{-\bigg(\partial_yu_1+D_xv_1 \bigg)
+i\bigg(D_xu_1-\partial_yv_1 \bigg)\bigg\}|x|^{2\lambda}(dx\wedge dy).
\end{split}\end{equation}
Notice that $u^2(x, y)-v^2(x, y)$ is even in $x$, and $u(x, y)v(x, y)$ is odd in $x$.
If u(x, y) is even in x and v(x, y)  odd in x, by Formulas\,(\ref{a c r1},\,\ref{intergal}), we could have
\begin{eqnarray}\label{uu1}
\int_{\partial\Omega}F(z)^2|x|^{2\lambda}dz=\int_{\Omega} \frac{4\lambda}{x}v^2 |x|^{2\lambda}(dx\wedge dy).
\end{eqnarray}
If u(x, y) is odd in x and v(x, y)  even in x, Formulas\,(\ref{a c r2},\,\ref{intergal})\,lead to
\begin{eqnarray}\label{uu2}
\int_{\partial\Omega}F(z)^2|x|^{2\lambda}dz=\int_{\Omega} -\frac{4\lambda}{x}u^2 |x|^{2\lambda}(dx\wedge dy).
\end{eqnarray}
As $\Omega$ is a bounded  domain symmetric in x, then Formulas\,(\ref{uu1},\,\ref{uu2}) lead to the following fact:
\begin{eqnarray*}
\int_{\partial\Omega}F(z)^2|x|^{2\lambda}dz=0, \ \ \hbox{when}\,F(z)\,\hbox{is\,$\lambda$-analytic}.
\end{eqnarray*}
This proves the Proposition.
\end{proof}

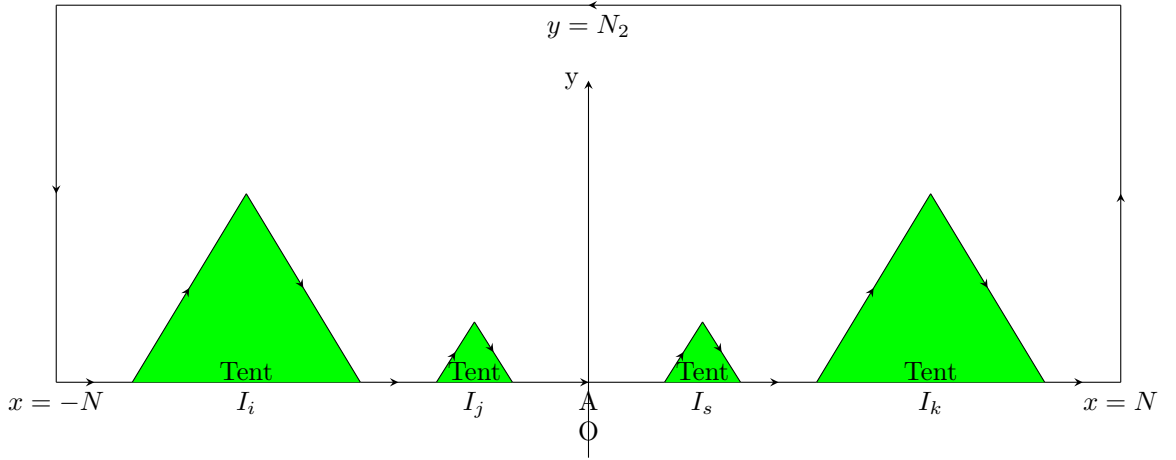
\tikzset
 {every pin/.style = {pin edge = {<-}},    
  > = stealth,                            
  flow/.style =    
   {decoration = {markings, mark=at position #1 with {\arrow{>}}},
    postaction = {decorate}
   },
  flow/.default = 0.5,          
  main/.style = {line width=1pt}                    
 }

\begin{tikzpicture}

\draw(1,0) -- (2,0);
\draw(3,0) -- (6,0);
\draw(-1,0) -- (-2,0);
\draw(-3,0) -- (-6,0);

\path[fill=green] (1,0) -- (2,0) -- (1.5,0.8) -- cycle;
\node [below]at(1.5,0.4){Tent};
\node [below]at(1.5,0){$I_{s}$};

\path[fill=green] (3,0) -- (6,0) -- (4.5,2.5) -- cycle;
\node [below]at(4.5,0.4){Tent};
\node [below]at(4.5,0){$I_{k}$};

\path[fill=green] (-3,0) -- (-6,0) -- (-4.5,2.5) -- cycle;
\node [below]at(-4.5,0.4){Tent};
\node [below]at(-4.5,0){$I_{i}$};

\path[fill=green] (-1,0) -- (-2,0) -- (-1.5,0.8) -- cycle;
\node [below]at(-1.5,0.4){Tent};
\node [below]at(-1.5,0){$I_{j}$};

\node [below]at(0,0){A};
\node [below]at(0,-0.4){O};
\node [below]at(7,0){$x=N$};
\node [below]at(-7,0){$x=-N$};
\node [below]at(0,5){$y=N_2$};

\draw [->](0,-1) -- (0,4);

\node [left]at(0,4){y};

\draw[flow](-1.5,0.8) --(-1,0) ;
\draw[flow](-2,0) --(-1.5,0.8) ;
\draw[flow](-3,0) -- (-2,0);
\draw[flow](-4.5,2.5) --(-3,0) ;
\draw[flow](-6,0) -- (-4.5,2.5);
\draw[flow](-7,0) --(-6,0) ;
\draw[flow](-7,5) --(-7,0) ;

\draw [flow](-1,0) -- (1,0);
\draw[flow](1,0) -- (1.5,0.8);
\draw[flow](1.5,0.8) -- (2,0);
\draw[flow](2,0) -- (3,0);
\draw[flow](3,0) -- (4.5,2.5);
\draw[flow](4.5,2.5) -- (6,0);
\draw[flow](6,0) -- (7,0);
\draw[flow](7,0) -- (7,5);

\draw[flow](7,5) -- (-7,5);

\end{tikzpicture}

Let $F(x, y)=u(x, y)+iv(x, y)\in H^p_\lambda(\RR^2_+)$ for $\frac{2\lambda}{2\lambda+1}< p \leq1$.  $F_t(x,y)=F(x, y+t)=u_t(x,y)+iv_t(x,y)$ for some fixed $t>0$, where $u_t(x,y)$ is  even or odd in x.  In the Picture above, the point $A$ denotes  $(0, t)$ and  the point $O$ designates the origin $(0, 0)$.  Let $u_{\nabla}^\ast(x)=\sup_{|s-x|<y}|u(s, y)|$, $u_t(x,y)=u(x, y+t)$, and $u_t(x)=u(x, t)$. Thus $(u_t)_{\nabla}^{*}(x)= \sup_{|s-x|<y}|u_t(s, y)|$.  Then $E_{\sigma}$ is defined as the  following set on a line:
$$E_{\sigma}=\bigg\{(x, y): (u_t)_{\nabla}^{*}(x)>\sigma,\ \ y\equiv t\bigg\},\ \ \hbox{and},\,|E_{\sigma}|_\lambda=\left|\bigg\{x: (u_t)_{\nabla}^{*}(x)>\sigma\bigg\}\right|_\lambda.$$
Notice that the set $E_{\sigma}$ is an open set on the real line $y=t$, thus we could write $E_{\sigma}$ as:
$$ E_{\sigma}= \bigcup_j I_j, $$
where $\left\{I_j\right\}_j$ are the open disjoint Euclidean intervals.  We use $T(I_j)$  to  denote as the tent: $$T(I_j)=\bigg\{ (x, y): |x-x_j|\leq r_j-y,\ \ y\geq t\bigg\},$$
where $x_j$ is the center of the interval $I_j$ and $r_j$  the radius of the interval $I_j$: $$I_j=\bigg\{ (x, y): x_j-r_j<x<x_j+r_j,\ \  y\equiv t\bigg\}.$$
We use $\partial T(I_j)$ to denote as the boundary of the tent $T(I_j)$. Then
$\Gamma$ is denoted as $$\Gamma=\bigg\{(x, y): (x, y)\in\bigcup_j\bigg(\partial T(I_j)\backslash I_j\bigg)\bigcup\bigg(\big\{(x, y): x\in\RR, y=t\big\} \backslash E_{\sigma}\bigg)\bigg\},$$ and $$\Gamma_N=\bigg\{(x, y): \Gamma\cap\big\{(x, y):-N<x <N, y\in\RR\big\}\bigg\}.$$
We use $(x_{\Gamma},y_{\Gamma})$ to denote as the point on $\Gamma$: $$\bigg\{(x_{\Gamma},y_{\Gamma})\bigg\}=\Gamma.$$
We set the domain  $\Omega$ as:
$$\Omega=\bigg\{(x,y):-N\leq x\leq N, y_{\Gamma}\leq y\leq N_2, \hbox{where}\,(x_{\Gamma},y_{\Gamma})\in\Gamma\bigg\},$$
where $N$ and $N_2$ are fixed positive numbers independent on  $t$ that we will discuss in Proposition\,\ref{7}. Notice that $u_t(x,y)$ is an even or odd function in x, thus $\Omega$ is a bounded domain symmetric in x on the upper half plane.

\begin{proposition}\label{7}
 Let $F(x, y)=u(x, y)+iv(x, y)\in H^p_\lambda(\RR^2_+)$ for $\frac{2\lambda}{2\lambda+1}< p \leq1$. Let $t>0$ to be  fixed, $F_t(x,y)=F(x, y+t)=u_t(x,y)+iv_t(x,y)$, where $u_t(x,y)$ is an even or odd function in x. Let $\Gamma$ and  $\Omega$ to be defined as above. Then we could obtain
$$ \int_{\Gamma}F_t(x,y)^{2} |x|^{2\lambda}dz=0 .$$
\end{proposition}
\begin{proof}
 It is clear that $F_{t}$ is a $\lambda$-analytic and continuous function. Taking $l=2, k=2, \delta=1/p-1/2$,
 we could have
 \begin{eqnarray}\label{i 5}
 \delta(2\lambda+1)-\frac{1}{k}&=&(2\lambda+1)\left(\frac{1}{p}-\frac{1}{2}\right)-\frac{1}{2}
 \\ \nonumber&=&\frac{2\lambda+1}{p}-(\lambda+1)
 \\ \nonumber&>&0.
 \end{eqnarray}
 Thus by Formulas\,(\ref{i 1},\,\ref{i 5}),  we could obtain
 \begin{eqnarray}\label{i 4}
 c\|F\|_{H_{\lambda}^p(\RR_+^2)}&\geq& \left(\int_0^{+\infty} y^{2(1/p-1/2)(1+2\lambda)-1}\left(\int_{\RR}|F(x,y)|^{2}|x|^{2\lambda}dx \right)dy \right)^{\frac{1}{2}}
 \\ \nonumber &\geq& \left(\int_t^{+\infty} y^{2(1/p-1/2)(1+2\lambda)-1}\left(\int_{\RR}|F(x,y)|^{2}|x|^{2\lambda}dx \right)dy \right)^{\frac{1}{2}}
 \\ \nonumber &\geq& \left(\int_0^{+\infty} t^{2(1/p-1/2)(1+2\lambda)-1}\left(\int_{\RR}|F(x,y+t)|^{2}|x|^{2\lambda}dx \right)dy \right)^{\frac{1}{2}}.
 \end{eqnarray}
 Then by Formula\,(\ref{i 4}),  the following holds:
\begin{eqnarray}\label{i 3}
\left(\int_0^{+\infty} \left(\int_{\RR}|F_t(x,y)|^{2}|x|^{2\lambda}dx \right)dy \right)^{\frac{1}{2}}\leq c t^{(1/2-1/p)(2\lambda+1)+(1/2)}\|F\|_{H_{\lambda}^p(\RR_+^2)}.
\end{eqnarray}
We use  $g_{t}(x)$ to denote as $$g_{t}(x)=\int_{\RR}\left(|F_t(x,y)|^{2}+|F_t(-x,y)|^{2}\right) |x|^{2\lambda}dy.$$
Then by Fubini's theorem and Formula\,(\ref{i 3}), we could deduce that  $$\left(\int_{\RR}g_{t}(x)dx\right)^{\frac{1}{2}}\leq c t^{(1/2-1/p)(2\lambda+1)+(1/2)}\|F\|_{H_{\lambda}^p(\RR_+^2)}.$$
Thus for any $\varepsilon>0$, there exists $N(\varepsilon)>0$, such that $$\int_{N(\varepsilon)}^{+\infty} g_{t}(x)dx <\varepsilon/2.$$
Notice that $g_{t}(x)$ is continuous in $x$, thus we could deduce that there exists $ N\geq N(\varepsilon)$, such that $g_{t}(N)<\varepsilon/2$:
\begin{eqnarray}\label{i4}
\int_0^{+\infty}\left(|F_t(N,y)|^{2}+|F_t(-N,y)|^{2}\right) |N|^{2\lambda}dy<\varepsilon/2.
\end{eqnarray}
We could deduce  from Formula\,(\ref{i 2})\,that for any $\varepsilon>0$, there exists $N_2$ such that the following holds:
\begin{eqnarray}\label{i5}
\int_{\RR}|F_t(x,N_2)|^{2} |x|^{2\lambda}dx<\varepsilon/2.
\end{eqnarray}
Notice that $u_t(x,y)$ is an even or odd function in x, thus $\Omega$ is the domain symmetric in x on the upper half plane: $$\Omega=\bigg\{(x,y):-N\leq x\leq N, y_{\Gamma}\leq y\leq N_2, \hbox{where}\,(x_{\Gamma},y_{\Gamma})\in\Gamma\bigg\}.$$
We use $y_{\Gamma}(x_0)$ to denote as $$y_{\Gamma}(x_0)=y_{\Gamma}\,\ \  \hbox{when}\,\ \  x_{\Gamma}=x_0 .$$
Then   $y_{\Gamma}(x_0)=y_{\Gamma}(-x_0)$. By Formula\,(\ref{i})\,we have:
\begin{equation}\label{i 6}
\begin{split}
\left|\int_{\partial\Omega}F_t(x,y)^{2} |x|^{2\lambda}dz\right|
&=\left|\int_{\Gamma_{N}}F_t(x,y)^{2} |x|^{2\lambda}dz +\int_{y_{\Gamma}(N)}^{N_2}F_t(N,y)^{2}|N|^{2\lambda}dy\right.
\\&+\left.\int_{N}^{-N}F_t(x,N_2)^{2} |x|^{2\lambda}dx +\int_{N_2}^{y_{\Gamma}(-N)}F_t(-N,y)^{2}|N|^{2\lambda}dy\right|
=0.\end{split}
\end{equation}
From  Formulas\,(\ref{i4},\,\ref{i5},\,\ref{i 6}), we could deduce the following inequality  for any $\varepsilon>0$:
\begin{eqnarray*}
\left|\int_{\Gamma_{N}}F_t(x,y)^{2} |x|^{2\lambda}dz\right|\lesssim \varepsilon.
\end{eqnarray*}
Notice  that $N(\varepsilon)\rightarrow+\infty$ as $\varepsilon\rightarrow0+$.
By the arbitrariness of $\varepsilon>0$, then we could deduce that:
\begin{eqnarray}\label{ii}
\int_{\Gamma}F_t(x,y)^{2} |x|^{2\lambda}dz=\lim_{N\rightarrow+\infty}\left|\int_{\Gamma_{N}}F_t(x,y)^{2} |x|^{2\lambda}dz\right|=0.
\end{eqnarray}
This proves the Proposition.
\end{proof}

\begin{proposition}
$F(x, y)=u(x, y)+iv(x, y)\in H^p_\lambda(\RR^2_+)$ for $\frac{2\lambda}{2\lambda+1}< p\leq1$. For any fixed $ t>0$, $F_t(x,y)$ is defined as $F_t(x,y)=F(x, y+t)=u_t(x,y)+iv_t(x,y)$, $u_t(x)=u(x, t)$, $v_t(x)=v(x, t)$, $E_{\sigma}=\big\{(x, y): (u_t)_{\nabla}^{*}(x)>\sigma,\ \ y\equiv t\big\}$,  where $u_t(x,y)$ is an even or odd function in x.    Then we could obtain
\begin{eqnarray}\label{iii}
|\left\{x\in\RR: |v_t(x)|\geq\sigma\right\}|_{\lambda}\leq 3|E_{\sigma}|_{\lambda}+\frac{2}{\sigma^2} \int_{0}^{\sigma}s|E_{s}|_{\lambda}ds.
\end{eqnarray}
\end{proposition}

\begin{proof}
Let $\Gamma$ to be  defined as above. Then from Formula\,(\ref{ii}), we have
\begin{eqnarray*}
\int_{\Gamma}(u_t(x, y)+iv_t(x, y))^{2} |x|^{2\lambda}dz=0.
\end{eqnarray*}
Taking the real part of the above equation, we could get
\begin{eqnarray*}
Re\int_{\Gamma}(u_t(x, y)+iv_t(x, y))^{2} |x|^{2\lambda}dz=Re\int_{\Gamma}(u_t(x, y)+iv_t(x, y))^{2} |x|^{2\lambda}(dx+idy)=0.
\end{eqnarray*}
That is
\begin{eqnarray*}
\int_{\Gamma}(u_t(x, y)^2-v_t(x, y)^2)|x|^{2\lambda}dx-2u_t(x, y)v_t(x, y)|x|^{2\lambda}dy=0.
\end{eqnarray*}
Then we could obtain:
\begin{eqnarray}\label{5}
\nonumber0&=&\int_{\big\{(x, y): x\in\RR, y=t\big\} \backslash E_{\sigma}}(u_t(x)^2-v_t(x)^2)|x|^{2\lambda}dx +\int_{\cup_j\left(\partial T(I_j)\backslash I_j\right)}(u_t(x, y)^2-v_t(x, y)^2)|x|^{2\lambda}dx\\&-& \int_{\cup_j\left(\partial T(I_j)\backslash I_j\right)} 2u_t(x, y)v_t(x, y)|x|^{2\lambda}dy.
\end{eqnarray}
It is clear that
\begin{eqnarray}\label{6}
\left|\int_{\cup_j\left(\partial T(I_j)\backslash I_j\right)} 2u_t(x, y)v_t(x, y)|x|^{2\lambda}dy\right|\leq\int_{\cup_j\left(\partial T(I_j)\backslash I_j\right)}(u_t(x, y)^2+v_t(x, y)^2)|x|^{2\lambda}dy.
\end{eqnarray}
By Formulas\,(\ref{5},\,\ref{6}),  we could have
\begin{eqnarray*}
\int_{\big\{(x, y): x\in\RR, y=t\big\} \backslash E_{\sigma}}v_t(x)^2|x|^{2\lambda}dx \leq \int_{\big\{(x, y): x\in\RR, y=t\big\} \backslash E_{\sigma}}u_t(x)^2|x|^{2\lambda}dx +\int_{\cup_j\left(\partial T(I_j)\backslash I_j\right)}2u_t(x, y)^2|x|^{2\lambda}dx.
\end{eqnarray*}
Notice that $|u_t(x, y)|\leq \sigma$ on $\cup_j\left(\partial T(I_j)\backslash I_j\right)$, thus we could obtain that
\begin{eqnarray}\label{measure}
\int_{\big\{(x, y): x\in\RR, y=t\big\} \backslash E_{\sigma}}v_t(x)^2|x|^{2\lambda}dx \leq \int_{\big\{(x, y): x\in\RR, y=t\big\} \backslash E_{\sigma}}((u_t)_{\nabla}^{*}(x))^2|x|^{2\lambda}dx +2\sigma^2|E_{\sigma}|_{\lambda}.
\end{eqnarray}
We could also notice that:
\begin{eqnarray}\label{i0}
 & &\ \ \int_{\big\{(x, y): x\in\RR, y=t\big\} \backslash E_{\sigma}}((u_t)_{\nabla}^{*}(x))^2|x|^{2\lambda}dx
 \\ \nonumber &=&2\int_0^{+\infty}s\left|\bigg\{x\in\big\{(x, y): x\in\RR, y=t\big\} \backslash E_{\sigma}:(u_t)_{\nabla}^*(x)>s\bigg\}\right|_{\lambda}ds
\\ \nonumber&\leq&2\int_0^{\sigma}s|E_s|_{\lambda}ds.
\end{eqnarray}
Then by Formulas\,(\ref{measure},\,\ref{i0}), we could obtain the following Formula:
\begin{eqnarray*}
|\left\{x\in\RR: |v_t(x)|\geq\sigma\right\}|_{\lambda}&\leq& |E_{\sigma}|_{\lambda}+\left|\bigg\{x\in\big\{(x, y): x\in\RR, y=t\big\} \backslash E_{\sigma}: |v_t(x)|\geq\sigma\bigg\}\right|_{\lambda}\\&\leq&|E_{\sigma}|_{\lambda}+\sigma^{-2}\int_{\big\{(x, y): x\in\RR, y=t\big\} \backslash E_{\sigma}}v_t(x)^2|x|^{2\lambda} dx
\\&\leq&|E_{\sigma}|_{\lambda}+\sigma^{-2}\int_{\big\{(x, y): x\in\RR, y=t\big\} \backslash E_{\sigma}}((u_t)_{\nabla}^{*}(x))^2|x|^{2\lambda}dx +2|E_{\sigma}|_{\lambda}
\\&\leq&3|E_{\sigma}|_{\lambda}+\frac{2}{\sigma^2} \int_{0}^{\sigma}s|E_{s}|_{\lambda}ds.\\
\end{eqnarray*}
\end{proof}

\begin{theorem}\label{iv}
 $F(x, y)=u(x, y)+iv(x, y)\in H^p_\lambda(\RR^2_+)$  $(\frac{2\lambda}{2\lambda+1}< p\leq1)$,  where $u(x,y)$ is an even or odd function in x. We use $F_t(x,y)$ to denote as  $F_t(x,y)=u_t(x,y)+iv_t(x,y)=F(x, y+t)$ for any fixed $ t>0$ . Then
\begin{eqnarray}\label{i01}
\sup_{y>0}\int_{-\infty}^{+\infty} |v(x, y)|^p|x|^{2\lambda}dx\leq c\|u_{\nabla}^*\|_{L^p_{\lambda}}^p.
\end{eqnarray}
\end{theorem}
\begin{proof}
By Formula\,(\ref{iii}),  for any fixed $ t>0$,  we could deduce the following Formula:
\begin{eqnarray*}
\int_{-\infty}^{+\infty}|v(x, t)|^p|x|^{2\lambda}dx&=&\int_{-\infty}^{+\infty}|v_t(x)|^p|x|^{2\lambda}dx\\
&=&\int_{0}^{+\infty}p\sigma^{p-1}|\{x\in\RR:|v_t(x)|>\sigma\}|_{\lambda}d\sigma\\
&\leq&\int_{0}^{+\infty}3p\sigma^{p-1}|E_{\sigma}|_{\lambda}d\sigma+\int_{0}^{+\infty}2p\sigma^{p-3}\int_0^{\sigma}s|E_s|_{\lambda}dsd\sigma \\
&=&3\int_{0}^{+\infty}|(u_t)_{\nabla}^*(x)|^p|x|^{2\lambda}dx+\int_{0}^{+\infty}\left(2p\int_s^{+\infty}\sigma^{p-3}d\sigma\right) s|E_s|_{\lambda}ds\\
&\leq&3\|u_{\nabla}^*\|_p^p+\frac{2p}{2-p}\int_{0}^{+\infty}s^{p-1}|E_s|_{\lambda}ds\\
&\leq&\frac{8-3p}{2-p}\|u_{\nabla}^*\|_{L^p_{\lambda}}^p.
\end{eqnarray*}
Then
$$\sup_{y>0}\int_{-\infty}^{+\infty} |v(x, y)|^p|x|^{2\lambda}dx\leq c\|u_{\nabla}^*\|_{L^p_{\lambda}}^p.$$
\end{proof}
We set $u_o$, $u_e$, $v_o$ and $v_e$ as: $u_o=(u(x,y)-u(-x,y))/2$, $u_e=(u(x,y)+u(-x,y))/2$, $v_o=(v(x,y)-v(-x,y))/2$, $v_e=(v(x,y)+v(-x,y))/2$. Then the functions $F_o$ and $F_e$ can be given by $$F_o=u_o+iv_e,\ \ \ F_e=u_e+iv_o.$$
\begin{proposition}\label{vv}
If $F=u+iv$ is a $\lambda$-analytic function,  then $F_o=u_o+iv_e, F_e=u_e+iv_o$ are $\lambda$-analytic functions, furthermore we could have
$$F\in H^p_{\lambda}(\RR_+^2)\Leftrightarrow F_e,\,F_o\in H^p_{\lambda}(\RR_+^2),$$
and $$\|F\|_{H^p_{\lambda}(\RR_+^2)}^p\thickapprox\|F_e\|_{H^p_{\lambda}(\RR_+^2)}^p+\|F_o\|_{H^p_{\lambda}(\RR_+^2)}^p,$$
where $\frac{2\lambda}{2\lambda+1}<p\leq1$.
\end{proposition}
\begin{proof}
It is clear that $F_o=u_o+iv_e, F_e=u_e+iv_o$ both satisfy the the $\lambda$-Cauchy-Riemann equations
\begin{eqnarray*}\left\{\begin{array}{ll}
                                    D_xu_o-\partial_y v_e=0,&  \\
                                    \partial_y u_o +D_xv_e=0,&
                                 \end{array}\right.
 \left\{\begin{array}{ll}
                                    D_xu_e-\partial_y v_o=0,&  \\
                                    \partial_y u_e +D_xv_o=0.&
                                 \end{array}\right.
\end{eqnarray*}
Notice that $$F=F_e+F_o,$$
thus we could obtain that
$$\|F\|_{H^p_{\lambda}(\RR_+^2)}^p\leq \|F_e\|_{H^p_{\lambda}(\RR_+^2)}^p+\|F_o\|_{H^p_{\lambda}(\RR_+^2)}^p.$$
Also we could obtain the following Formula:
\begin{eqnarray*}
\|F_e\|_{H^p_\lambda(\RR^2_+)}^p=\sup\limits_{y>0}\left\{c_{\lambda}\int_{\RR}|u_e^2+v_o^2|^{p/2}|x|^{2\lambda}dx \right\}\leq c_p\sup\limits_{y>0}\left\{c_{\lambda}\int_{\RR}|u^2+v^2|^{p/2}|x|^{2\lambda}dx \right\},
\end{eqnarray*}
and
\begin{eqnarray*}
\|F_o\|_{H^p_\lambda(\RR^2_+)}^p=\sup\limits_{y>0}\left\{c_{\lambda}\int_{\RR}|u_o^2+v_e^2|^{p/2}|x|^{2\lambda}dx \right\}\leq c_p\sup\limits_{y>0}\left\{c_{\lambda}\int_{\RR}|u^2+v^2|^{p/2}|x|^{2\lambda}dx \right\}.
\end{eqnarray*}
This proves the Proposition.
\end{proof}

\begin{theorem}\label{sss}
 If $F(z)=u(x, y)+iv(x, y) \in H_{\lambda}^p(\RR_+^2)$, for $\frac{2\lambda}{2\lambda+1}< p\leq1$, then the following holds:
\begin{equation}\label{sk5}
\|F\|_{H_{\lambda}^p(\RR_+^2)}\leq c\|u_{\nabla}^*\|_{L^p_{\lambda}}.
\end{equation}
\end{theorem}
\begin{proof}
By Theorem\,\ref{iv} and Proposition\,\ref{vv}, we could get
\begin{equation}\begin{split}
\|F\|^p_{H_{\lambda}^p(\RR_+^2)} & =\displaystyle{\sup_{t>0}\int_{\RR}|u(x, t)^2+v(x, t)^2|^{p/2}|x|^{2\lambda}dx}\\
& \leq \displaystyle{\sup_{t>0}\int_{\RR}\left(|u(x, t)|^2+|v_o(x, t)+v_e(x, t)|^2\right)^{p/2}|x|^{2\lambda}dx} \\
& \leq \displaystyle{c\sup_{t>0}\int_{\RR}(|u(x, t)|^p+|(u_o)_{\nabla}^*(x)|^p+|(u_e)_{\nabla}^*(x)|^p)|x|^{2\lambda}dx} \\
& \leq \displaystyle{c\int_{\RR}|u_{\nabla}^*(x)|^p|x|^{2\lambda}dx}.
\end{split}\end{equation}
Thus we could deduce Formula\,(\ref{sk5}) holds.
\end{proof}

Thus by Theorem\,\ref{sss} and Proposition\,\ref{ss}(iv), we could deduce the following Theorem:
\begin{theorem}\label{ssst}
 If $F(z)=u(x, y)+iv(x, y) \in H_{\lambda}^p(\RR_+^2)$, for $\frac{2\lambda}{2\lambda+1}< p\leq1$, then the following holds:
\begin{equation}\label{sk5t}
\|F\|_{H_{\lambda}^p(\RR_+^2)}\sim \|u_{\nabla}^*\|_{L^p_{\lambda}}.
\end{equation}
\end{theorem}

\end{document}